\documentclass[10pt]{amsart}

\usepackage{amssymb,amsmath,amsthm}
\usepackage{tikz}
\usepackage{tikz-cd}
\usepackage[arc,all]{xy}
\usepackage{eucal}

\usepackage{color}

\usepackage{comment}

\theoremstyle{plain}
\newtheorem{cor}[subsection]{Corollary}
\newtheorem{lem}[subsection]{Lemma}
\newtheorem{thm}{Theorem}
\newtheorem{prop}[subsection]{Proposition}

\theoremstyle{definition}

\theoremstyle{remark}
\newtheorem{rem}[subsection]{Remark}

\newcommand{\Z}{\mathbb{Z}}
\newcommand{\C}{\mathbb{C}}
\newcommand{\lp}{_{(p)}}
\newcommand{\Zp}{\mathbb{Z}_{(p)}}

\newcommand{\de}{\delta}
\newcommand{\opn}{\operatorname}

\numberwithin{equation}{section}

\title[The $p$-primary subgroup of $H^*(BPU_n)$]{The $p$-primary subgroups of the cohomology of $BPU_n$ in dimensions less than $2p+5$}

\author{Xing Gu}
\author{Yu Zhang}
\author{Zhilei Zhang}
\author{Linan Zhong$^*$}

\address{Center for Topology and Geometry based Technology, School of Mathematical Sciences, Hebei Normal University}

\email{gux2006@outlook.com}

\address{Department of Mathematics, Nankai University}

\email{zhang.4841@osu.edu}

\address{Department of Mathematics, Nankai University}

\email{15829207515@163.com}

\address{Department of Mathematics, Yanbian University}

\email{lnzhong@ybu.edu.cn}

\subjclass[2020]{55T10, 55R35, 55R40}

\thanks{}

\thanks{$*$ Corresponding author}

\begin{document}

\maketitle

\begin{abstract}
    Let $PU_n$ denote the projective unitary group of rank $n$ and $BPU_n$ be its classifying space. For an odd prime $p$, we extend previous results to a complete description of $H^s(BPU_n;\Z)\lp$ for $s<2p+5$ by showing that the $p$-primary subgroups of $H^s(BPU_n;\Z)$ are trivial for $s = 2p+3$ and $s = 2p+4$. 
\end{abstract}

\section{Introduction}\label{sec:intro}

Let $U_n$ denote the group of $n \times n$ unitary matrices. The unit circle $S^1$ can be viewed as the normal subgroup of scalar matrices of $U_n$.  We let $PU_n$ denote the quotient group of $U_n$ by $S^1$, and $BPU_n$ be the classifying space of $PU_n$. In this paper we consider $H^*(BPU_n;\Z)$, th7e ordinary cohomology of $BPU_n$ with coefficients in $\Z$.

\subsection*{A review of the literature} The ordinary and generalized cohomology of $BPU_n$ for special $n$ has been the subject of various works such as Kono-Mimura \cite{kono1975cohomology}, Kameko-Yagita \cite{Kameko2008brown}, Kono-Yagita \cite{kono1993brown}, Toda \cite{toda1987cohomology}, and Vavpeti{\v{c}}-Viruel \cite{vavpetivc2005mod}. Vezzosi \cite{vezzosi2000chow} and Vistoli \cite{vistoli2007cohomology} studied the Chow ring of the classifying space (in the sense of Totaro \cite{totaro1999chow}) of $BPGL_3(\C)$ and $BPGL_p(\C)$ for $p$ an odd prime, respectively. Much of their results applies to the ordinary cohomology of $BPU_p$. 

None of the works above dealt with $H^*(BPU_n;\Z)$ for $n$ not a prime number. The first named author considered $H^*(BPU_n;\Z)$, as well as the Chow ring of $BPGL_n(\C)$ for an arbitrary $n$ in \cite{gu2020almost}, \cite{gu2019cohomology} and  \cite{gu2019some}. In particular, in \cite{gu2019cohomology}, the first named author determined the ring structure of $H^*(BPU_n;\Z)$ in dimensions less than or equal to $10$.

Other related works include Duan \cite{duan2017cohomology}, in which the integral cohomology of $PU_n$ is fully determined, and Crowley-Gu \cite{crowley2021h}, in which the image of the canonical map $H^*(BPU_n;\Z)\to H^*(BU_n;\Z)$ is studied.

The cohomology of $BPU_n$ plays significant roles in the study of the topological period-index problem (\cite{antieau2014period}, \cite{antieau2014topological}, \cite{gu2019topological} and \cite{gu2020topological}), and in the study of anormalies in partical physics (\cite{cordova2020anomalies}, \cite{garcia2019dai}).

\subsection*{Notations} Throughout the rest of this paper, $H^* (-)$ denotes $H^*(-;\Z)$. For an abelian group $A$ and a prime number $p$, let $A\lp$ be the localization of $A$ at $p$, and let $_pA$ denotes the $p$-primary subgroup of $A$, i.e., the subgroup of $A$ of all torsion elements with torsion order a power of $p$. In particular, we have a canonical isomorphism $_pH^*(-)\cong {_p[H^*(-)\lp]}$, and we will not distinguish the two throughout this paper. Tensor products of $\Z\lp$-modules are always taken over $\Z\lp$.

\subsection*{The main theorem and some remarks} We review a basic fact on the cohomology of $BPU_n$. Consider the short exact sequence of Lie groups
\begin{equation*}
    1 \to \Z/ n \to SU_n \to PSU_n \simeq PU_n \to 1,
\end{equation*}
which induces a fiber sequence of their classifying spaces
\begin{equation}\label{eq:Bn_cover}
    B(\Z/ n) \to BSU_n \to BPU_n
\end{equation}

When $p \nmid n$, the space $B(\Z/ n)$ is $p$-locally contractible and we have
\begin{equation}
    H^*(BPU_n;\Z\lp) \cong H^{*}(BSU_{n};\Z\lp)
\end{equation}
Since $\Z\lp$ is a flat $\Z$-module, and in particular, $H^*(-;\Z\lp)\cong H^*(-)\lp$, we have an isomorphism of $\Zp$-algebras
\begin{equation}\label{equation: cohomology of BSUn}
    H^*(BPU_n)\lp \cong H^{*}(BSU_{n})\lp = \Zp[c_2,c_3,\dots,c_n],
\end{equation}
which shows $H^*(BPU_n)\lp$ is torsion-free for $p \nmid n$.  In other words, we have the following
\begin{prop}\label{pro:n-torsion}
    Suppose $x\in H^*(BPU_n)$ is a torsion class. Then there exists some $i\geq 0$ such that $n^ix = 0$.
\end{prop}

Therefore, to determine the graded abelian group structure of $H^s(BPU_n)$, it suffices to consider the $p$-primary subgroup $_pH^s(BPU_n)$ for $p\mid n$. 

\begin{rem}\label{rem:Vezzosi}
    In the case of Chow rings, Vezzosi \cite{vezzosi2000chow} proved the stronger result that all torsion classes in the Chow ring of $BPGL_n(\mathbb{C})$ are $n$-torsion.
\end{rem}

To state the main theorem, recall that, as shown in \cite{gu2019cohomology}, the integral cohomology group $H^3(BPU_n)$ is generated by a class denoted by $x_1$. In addition, $\opn{P}^i$ will denote the $i$th Steenrod reduced power operation, and   
\[\delta: H^*(-,\Z/p)\to H^{* + 1}(-)\]
will denote the connecting homomorphism. Finally, a bar over an integral cohomology class will denote the mod $p$ reduction of this class. For instance, $\bar{x}_1$ denotes the mod $p$ reduction of $x_1$, which is in $H^3(BPU_n;\Z/p)$.

\begin{thm}\label{thm: main thm, 2p+3 and 2p+4 have trivial p torsion}
Let $p > 2$ be a prime number, and $n=p^rm$ for a positive integer $m$ co-prime to $p$. Then the $p$-primary subgroup of $H^{s}(BPU_n)$ in dimensions less than $2p+5$ is as follows:
\begin{enumerate}
    \item For $r > 0$, we have 
    \begin{equation*}
    _pH^s(BPU_n)\cong
    \begin{cases}
        \Z/p^r,\ s=3,\\
        \Z/p,\ s=2p+2,\\
        0,\  s<2p+5,\ s\neq 3, 2p+2.
    \end{cases}
    \end{equation*}
    The group $_pH^{2p + 2}(BPU_n)$ is generated by $\delta\opn{P}^1(\bar{x}_1)$. 
    \item For $r = 0$, we simply have $_pH^s(BPU_n) = 0$ for all $s\geq 0$. 
\end{enumerate}

\end{thm}
\begin{rem}
    Note $_pH^s(BPU_n)\cong {_pH}^s(BPU_n)\lp$. By the discussion preceding Remark \ref{rem:Vezzosi}, Theorem \ref{thm: main thm, 2p+3 and 2p+4 have trivial p torsion} completely determines $H^s(BPU_n;\Z\lp)$ for $0 \leq s< 2p + 5$. 
\end{rem}

For $s\leq 3$, the groups $_pH^s(BPU_n)$ are well known and are part of Theorem 1.1 of \cite{gu2019cohomology}. For $3 < s < 2p+2$, they are given in Theorem 1.2 of \cite{gu2019cohomology}. Therefore, what remains to show is 
\begin{equation}\label{eq:remain_to_show}
    _pH^{2p+2}(BPU_n) \cong \Z/p,\ _pH^{2p+3}(BPU_n) =  {_pH}^{2p+4}(BPU_n) =0.
\end{equation}

\begin{rem}
    For $p = 2$,  it was shown by the first named author \cite{gu2019cohomology} that the $2$-torsion subgroup of $H^{s}(BPU_n)$ in dimension $s = 2p+3 = 7$ is $\Z/2$ if $n \equiv 2$ mod $4$, and is $0$ otherwise. In particular, Theorem \ref{thm: main thm, 2p+3 and 2p+4 have trivial p torsion} does not generalize to the case $p=2$.
\end{rem}

\begin{rem}
    For $p = 3$, \eqref{eq:remain_to_show} follows immediately from the computation in \cite{gu2019cohomology} of $H^s(BPU_n)$ in dimensions $8, 9$ and $10$. 
\end{rem}

\subsection*{Organization of the paper} In Section \ref{sec:spectral_sequence}, we discuss some preliminary results of the Serre spectral sequence $^UE$ associated to the fiber sequence
$U: ~ BU_n \to BPU_n \to K(\mathbb{Z}, 3)$.  This will be our main tool for computing the $p$-primary subgroup $_pH^s(BPU_n)$.  We will also show that \eqref{eq:remain_to_show} can be deduced from Theorem 1.2 of \cite{gu2019cohomology} and Proposition \ref{prop: four term exact sequence}, which says that certain chain complex $\mathcal{M}$ constructed from the differentials in $^UE$ is exact.

In Section \ref{sec:exactness}, we prove Proposition \ref{prop: four term exact sequence}.  The proof is based on the explicit computation of some relevant differentials in $^UE$.  This section finishes our proof of Theorem \ref{thm: main thm, 2p+3 and 2p+4 have trivial p torsion}.

\subsection*{Acknowledgments}
The authors would like to thank Ben Williams for various helpful editorial suggestions. The first named author would like to thank the hospitality of Professor Jie Wu, as well as the Center for Topology and Geometry based Technology (CTGT) and the School of Mathematical Sciences at Hebei Normal University, and acknowledge the supports from the National Natural Science Foundation of China (No. 21113062) and from the High-level Scientific Research Foundation of Hebei Province (No. 13113093). The second named author and the third named author would like to thank Xiangjun Wang for helpful discussions.
The second named author and the third named author were supported by the National Natural Science Foundation of China 
(No. 11871284).  The fourth named author was supported by the National Natural Science Foundation of China (No. 12001474; 11761072).  All authors contribute equally.

\section{The spectral sequences}\label{sec:spectral_sequence}

\subsection*{The Serre spectral sequence $^UE$}

We follow the strategy employed in \cite{gu2019cohomology} to compute the cohomology of $BPU_n$. The short exact sequence of Lie groups
$$1 \to S^1 \to U_n \to PU_n \to 1$$
induces a fiber sequence of their classifying spaces
$$BS^1 \to BU_n \to BPU_n.$$
Notice that $BS^1$ is of the homotopy type of the Eilenberg-Mac Lane space $K(\Z, 2)$ and indeed we obtain another fiber sequence
\begin{equation}\label{eq:chi}
U: ~ BU_n \to BPU_n \xrightarrow{\chi} K(\Z, 3).
\end{equation}
\begin{rem}
    In general, it is not always possible to obtain a fiber sequence of the form $F\to E\to B$ from a fiber sequence $\Omega B\to F\to E$. See Ganea \cite{ganea1967induced} for more.
\end{rem}

We will use the Serre spectral sequence associated to the last fiber sequence to compute the cohomology of $BPU_n$. 
For notational convenience, we denote this spectral sequence by $^UE$. The $E_2$ page of 
$^UE$ has the form 
$$^U E^{s, t}_{2} = H^{s}(K(\mathbb{Z},3);H^{t}(BU_{n})) \Longrightarrow H^{s+t}(BPU_{n}).$$

In principle, Cartan and Serre \cite{cartan19551955} determined the cohomology of $K(A,n)$ for all finitely generated abelian groups $A$.  Also see Tamanoi \cite{tamanoi1999subalgebras} for a nice treatment.

We summarize the $p$-local cohomology of $K(\Z, 3)$ in low dimensions as follows.

\begin{prop}\label{prop: p local cohomology of KZ3 below 2p+5}
Let $p > 2$ be a prime. In degrees up to $2p+5$, we have 
\begin{equation}\label{equation: p local cohomology of KZ3 below 2p+5}
H^{s}(K(\mathbb{Z},3))_{(p)} = 
\begin{cases}
\Zp, & s = 0,\ 3,\\
\Z/p, &  s = 2p+2,\ 2p+5,\\
0, & s < 2p+5,  s \neq 0,\ 3,\ 2p+2.
\end{cases}
\end{equation}
where $x_1,~ y_{p, 0},~ x_1 y_{p, 0}$ are generators on degree $3, 2p+2, 2p+5$ respectively. In addition, we have $y_{p, 0} = \delta\opn{P}^1(\bar{x}_1)$.
\end{prop}
Here we use the same notations for the generators as in \cite{gu2019cohomology}. Sometimes we abuse notations and let $x_1, y_{p, 0}$ denote $\chi^*(x_1)$, $\chi^*(y_{p,0})$, where $\chi: BPU_n\to K(\Z,3)$ is defined in \eqref{eq:chi}. For instance, we have

\begin{prop}[Theorem 1.2, \cite{gu2019cohomology}]\label{prop:thm1_2}
    Let $p$ be a prime. In $H^{2p+2}(BPU_{n})$, we have $y_{p,0}\neq 0$ of order $p$ when $p\mid n$, and $y_{p,0}=0$ otherwise. Furthermore, the $p$-torsion subgroup of $H^k(BPU_n)$ is $0$ for $3<k<2p+2$.
\end{prop}

Also recall
\begin{equation}\label{equation: cohomology of BUn}
    H^{*}(BU_{n}) = \mathbb{Z}[c_1,c_2,\dots,c_n],\ |c_i|=2i.
\end{equation}
In particular, $H^{*}(BU_{n})$ is torsion-free.  We have 
$$^U E^{s, t}_{2} \cong H^{s}(K(\mathbb{Z},3)) \otimes H^{t}(BU_{n}).$$

\subsection*{The auxiliary fiber sequences and spectral sequences}

To determine some of the differentials in $^UE$, we consider two more fiber sequences. 

Let $T^n$ be the maximal torus of $U^n$ with the inclusion denoted by
\[\psi: T^n\to U_n.\]
Passing to quotients over $S^1$, we have another inclusion of maximal torus
\[\psi': PT^n\to PU_n.\]
The quotient map $T^n\to PT^n$ fits in an exact sequnce of Lie groups 
\[1\to S^1\to T^n\to PT^n\to 1,\]
which induces a fiber sequence
$$T: ~ BT^n \to BPT^n \to K(\mathbb{Z}, 3).$$
Notice that we have
\begin{equation}\label{equation: cohomology of BTn}
    H^{*}(BT^{n}) = \mathbb{Z}[v_1,v_2,\dots,v_n],\ |v_i|=2.
\end{equation}

The next fiber sequence is simply the path fibration for the space $K(\Z,3)$
$$K: ~ K(\mathbb{Z}, 2) \to * \to K(\mathbb{Z}, 3)$$
where $*$ denotes a contractible space. We denote their associated Serre spectral sequences as $^T E$ and $^K E$ respectively.

We denote the corresponding differentials
of $^UE$, ${^TE}$, and ${^KE}$ by ${^Ud}_*^{*,*}$, ${^Td}_*^{*,*}$, and ${^Kd}_*^{*,*}$, respectively, if there are risks of ambiguity. Otherwise, we simply denote the differentials by $d_*^{*,*}$.

These fiber sequences fit into the following homotopy commutative diagram:
\begin{equation}\label{eq:3_by_3_diag}
    \begin{tikzcd}
        K\arrow[d,"\Phi"]:& K(\Z,2)\arrow[r]\arrow[d,"B\varphi"]& *\arrow[r]\arrow[d]& 
        K(\Z,3)\arrow[d,"="]\\
        T:\arrow[d,"\Psi"]& BT^n\arrow[r]\arrow[d,"B\psi"]& BPT^n\arrow[r]\arrow[d,"B\psi'"]& K(\Z,3)\arrow[d,"="]\\
        U:& BU_n\arrow[r]& BPU_n\arrow[r]& K(\Z,3)
    \end{tikzcd}
\end{equation}
Here, the map $B\varphi: K(\Z,2)\simeq BS^1\to BT^n$ is the de-looping of the diagonal map $S^1\to T^n$. The induced homomorphism between cohomology rings is as follows:
\[B\varphi^*:H^*(BT^n) = \Z[v_1,v_2,\cdots,v_n] \to H^*(BS^1) = \Z[v],\ v_i\mapsto v.\]
The map $B\psi: BT^n\to BU_n$ induces the injective ring homomorphism
\begin{equation*}
    \begin{split}
        B\psi^*: H^*(BU_n) = \Z[c_1,\cdots,c_n] &\to H^*(BT^n) = \Z[v_1,\cdots,v_n],\\
        c_i &\mapsto \sigma_i(v_1,\cdots,v_n),
    \end{split}
\end{equation*}
where $\sigma_j(t_1,t_2,\cdots,t_n)$ be the $j$th elementary symmetric polynomial in variables $t_1,t_2,\cdots,t_n$:
\begin{equation}\label{eq:sigma_def}
    \begin{split}
        & \sigma_0(t_1,t_2,\cdots,t_n) = 1,\\
        & \sigma_1(t_1,t_2,\cdots,t_n) = t_1+t_2+\cdots+t_n,\\
        & \sigma_2(t_1,t_2,\cdots,t_n) = \sum_{i<j}t_it_j,\\
        & \vdots\\
        & \sigma_p(t_1,t_2,\cdots,t_n) = t_1t_2\cdots t_n.
    \end{split}
\end{equation}

We will use the associated maps of spectral sequences to compute the differentials in $^UE$.  This is possible because we have a good understanding of the corresonding differentials in $^TE$ and $^KE$.
In particular, we have the following results.

\begin{prop}[\cite{gu2019cohomology}, Corollary 2.16]\label{prop:diff0}
 The higher differentials of ${^KE}_{*}^{*,*}$ satisfy
 \begin{equation*}
 \begin{split}
   &d_{3}(v)=x_1,\\
   &d_{2p-1}(x_{1}v^{lp^{e}-1})=v^{lp^{e}-1-(p-1)}y_{p,0},\quad e > 0,\ \operatorname{gcd}(l,p)=1,\\
   &d_{r}(x_1)=d_{r}(y_{p,0})=0,\quad \textrm{for all }r,
 \end{split}
\end{equation*}
and the Leibniz rule. 
\end{prop}
\begin{rem}
    Proposition \ref{prop:diff0} is a special case of Corollary 2.16, \cite{gu2019cohomology}. Here, we take the opportunity to correct a typo in the original Corollary 2.16, \cite{gu2019cohomology}, where the condition $k \geq e$ should be replaced by $e > k$. 
\end{rem}

\begin{prop}[\cite{gu2019cohomology}, Proposition 3.2]\label{pro:diff1}
    The differential  $^{T}d_{r}^{*,*}$, is partially determined as follows:
\begin{equation}
 ^{T}d_{r}^{*,2t}(v_{i}^{t}\xi)={(B\pi_i)^{*}}({^Kd}_{r}^{*,2t}(v^{t}\xi)),
\end{equation}
where $\xi\in {^{T}E}_{r}^{*,0}$, a quotient group of $H^*(K(\mathbb{Z}, 3))$, and $\pi_i: T^{n}\rightarrow S^1$ is the projection of the $i$th diagonal entry. In plain words, $^{T}d_{r}^{*,2t}(v_{i}^{t}\xi)$ is simply $^{K}d_{r}^{*,2t}(v^{t}\xi)$ with $v$ replaced by $v_i$.
\end{prop}
\begin{rem}
    Here we correct another typo in the original Proposition 3.2 in \cite{gu2019cohomology}, in which ``$~\xi\in {^{T}E}_{r}^{0,*}$~'' should be replaced by ``~$\xi\in {^{T}E}_{r}^{*,0}$~''.
\end{rem}


\begin{prop}[\cite{gu2019cohomology}, Proposition 3.3]\label{prop:the_spectral_seq_TE}
    \begin{enumerate}
        \item The differential $^{T}d_{3}^{0,t}$ is given by the ``formal divergence''
        \[\nabla=\sum_{i=1}^{n}(\partial/\partial v_i): H^{t}(BT^{n};R)\rightarrow H^{t-2}(BT^{n};R),\]
        in such a way that $^{T}d_{3}^{0,*}=\nabla(-)\cdot x_{1}.$ For any ground ring $R=\mathbb{Z}$ or $\mathbb{Z}/m$ for any integer $m$.
        \item The spectral sequence degenerates at ${{^T}E}^{0,*}_{4}$. Indeed, we have $^{T}E_{\infty}^{0,*}={^{T}E_{4}}^{0,*}=\operatorname{Ker}{^Td}_{3}^{0,*}=\mathbb{Z}[v_{1}-v_{n},\cdots, v_{n-1}-v_{n}]$.
    \end{enumerate}
\end{prop}

\begin{cor}[\cite{gu2019cohomology}, Corollary 3.4]\label{cor:d3}
    \[^{U}d_{3}^{0,*}(c_{k})=\nabla(c_{k})x_1=(n-k+1)c_{k-1}x_1.\]
\end{cor}

\subsection*{Computations in the spectral sequence $^UE$}

In order to study
$$_pH^* (BPU_n)\cong {_p[H^*(BPU_n)\lp]},$$
it suffices to look at the $p$-localized spectral sequence, where the $E_2$ page becomes
\begin{equation}\label{equation: E2 tensor form}
    (^U E^{s, t}_{2})_{(p)} = H^{s}(K(\mathbb{Z},3))_{(p)} \otimes H^{t}(BU_{n}) = H^{s}(K(\mathbb{Z},3)) \otimes H^{t}(BU_{n})_{(p)}.
\end{equation}
By abuse of notation, for the rest of this paper, we let ${^U E}, {^T E}$ and $^K E$ denote the corresponding $p$-localized Serre spectral sequences.

By Proposition \ref{prop: p local cohomology of KZ3 below 2p+5} and \eqref{equation: cohomology of BUn}, in the range $s \leq 2p+5$, the only cases in which $^U E^{s, t}_{2}$ could be nonzero are when $s = 0, 3, 2p+2, 2p+5$ and $t$ is even. 

To simplify the notations, we let 
$$M^0 = ~ ^U E^{0,2p+2}_{2}, ~ M^1 = ~ ^U E^{3,2p}_{2}, ~ M^2 = ~ ^U E^{2p+2,2}_{2}, ~ M^3 = ~ ^U E^{2p+5,0}_{2}.$$

Inspection of degrees shows that $^U E^{3, 2p}_{*}$ can receive only the $d_3$ differential and support the $d_{2p-1}$ differential. Similarly, $^U E^{2p+2,2}_{*}$ can receive only the $d_{2p-1}$ differential and support the $d_3$ differential. In addition, all $d_2$'s are trivial and therefore we have $^U E^{*,*}_{2} = ~^U E^{*,*}_{3}$.

We let $\de^0$ be the map
$$\de^0: M^0 = ~^U E^{0,2p+2}_{3} \xrightarrow{d_3} ~^U E^{3, 2p}_{3} = M^1.$$
We let $\de^1$ be the composition
$$\de^1:M^1 = ~^U E^{3, 2p}_{3} \to ~^U E^{3, 2p}_{3}/ \opn{Im}d_3  = ~ ^U E^{3,2p}_{2p-1} \xrightarrow{d_{2p-1}} ~ ^U E^{2p+2,2}_{2p-1} = \opn{Ker} d_3  \subset M^2.$$
We let $\de^2$ be the map
$$\de^2:M^2 = ~^U E^{2p+2,2}_{3} \xrightarrow{d_3} ~^U E^{2p+5,0}_{3} = M^3.$$

One immediately sees that 
\[M^0 \xrightarrow{\de^0} M^1 \xrightarrow{\de^1} M^2 \xrightarrow{\de^2} M^3\]
is a chain complex of $\Zp$-modules, which we denote by $\mathcal{M}$. We will show later that Theorem \ref{thm: main thm, 2p+3 and 2p+4 have trivial p torsion} is a consequence of the following 
\begin{prop}\label{prop: four term exact sequence}
Let $p \geq 3$ be a prime number such that $p\mid n$.  The chain complex $\mathcal{M}$ defined above is exact.
\end{prop}

\begin{proof}[Proof of Theorem \ref{thm: main thm, 2p+3 and 2p+4 have trivial p torsion} assuming Proposition \ref{prop: four term exact sequence}]
Let $n = p^r m$. For $r = 0$, the theorem follows from Proposition \ref{pro:n-torsion}. In the rest of the proof we assume $r > 0$. First, we prove
$$_pH^{2p+2}(BPU_n) \cong \Z/p.$$
By Proposition \ref{prop:thm1_2}, $y_{p,0}\in ^UE_2^{2p+2,0}$ survives to a nonzero element in $H^{2p+2}(BPU_{n})$ of order $p$.  Therefore, we have 
$$^UE_{\infty}^{2p+2,0} = ~^UE_2^{2p+2,0}\cong\Z/p.$$
Since the only nontrivial entries in $^UE_2^{*,*}$ of total degree $2p + 2$ are $^UE_2^{2p+2,0}$ and $^UE_2^{0,2p+2}$, we have a short exact sequence of $\Zp$-modules
    \[0\to ~^UE_{\infty}^{2p+2,0}  \to H^{2p+2} (BPU_n)\lp \to ~^UE_{\infty}^{0,2p+2} \to 0.\]
    Since $^UE_{\infty}^{0,2p+2} \subset ~^UE_2^{0,2p+2}$ is a free $\Zp$-module, the above short exact sequence splits and we have
    \[H^{2p+2} (BPU_n)\lp \cong ~^UE_{\infty}^{2p+2,0} \oplus ~^UE_{\infty}^{0,2p+2},\]
    from which we deduce
    \[_pH^{2p+2} (BPU_n) \cong ~^UE_{\infty}^{2p+2,0} \cong \Z/p.\]
Since the row $E_{\infty}^{*,0}$ is the image of $\chi^*$, the above implies
\begin{equation}\label{eq:imageofchi}
	_pH^{2p+2} (BPU_n)  = \chi^*(H^{2p+2}(K(\Z,3))).
\end{equation}
From \eqref{eq:imageofchi} and Proposition \ref{prop: p local cohomology of KZ3 below 2p+5}, it follows that $_pH^{2p + 2}(BPU_n)$ is generated by $\delta\opn{P}^1(\bar{x}_1)$.
    
Next, we prove 
$$_pH^{2p+3}(BPU_n) = H^{2p+3}(BPU_n)\lp = 0.$$
The exactness of $\mathcal{M}$ at $M^1$ implies $^U E^{3, 2p}_{\infty} = 0$.  On the other hand, ${^UE}_2^{3,2p}$ is the only nontrivial entry in ${^UE}_2^{*,*}$ of total degree $2p+3$. Hence, we have 
    \[_pH^{2p+3}(BPU_n)\subset H^{2p+3}(BPU_n)\lp = {^UE}^{3, 2p}_{\infty} = 0.\]
Finally, we prove
$$_pH^{2p+4}(BPU_n) = 0.$$
The exactness of $\mathcal{M}$ at $M^2$ implies $^U E^{2p+2, 2}_{\infty} = 0$.  Since ${^UE}_2^{0,2p+4}$ and ${^UE}_2^{2p+2,2}$ are the only nontrivial entries in ${^UE}_2^{*,*}$ of total degree $2p+4$, we have
    \[H^{2p+4}(BPU_n)\lp\cong {^UE}_{\infty}^{0,2p+4},\]
which is torsion-free. In particular, we have $_pH^{2p+4}(BPU_n) = 0$.

\end{proof}

The proof of Proposition \ref{prop: four term exact sequence} occupies Section \ref{sec:exactness}.

\section{The proof of Proposition \ref{prop: four term exact sequence}}\label{sec:exactness}
From \eqref{equation: E2 tensor form}, we can write out the $\Zp$-modules $M^0, M^1, M^2, M^3$ more explicitly:
$$M^0 = H^{0}(K(\Z,3)) \otimes H^{2p+2}(BU_n)_{(p)}\cong H^{2p+2}(BU_n)_{(p)}$$
is the free $\Zp$-module generated by monomials in $c_1,\cdots,c_{p+1}$ in dimension $2p+2$, and
$$M^1 = H^{3}(K(\Z,3)) \otimes H^{2p}(BU_n)_{(p)} \cong H^{2p}(BU_n)_{(p)}$$
is the free $\Zp$-module generated by elements of the form $cx_1$ where $c$ is a monomial in $c_1,\cdots,c_p$ in dimension $2p$. Furthermore, we have
$$M^2 = H^{2p+2}(K(\Z,3)) \otimes H^{2}(BU_n)_{(p)} = \Zp \{c_{1}y_{p,0}\}/p\cong\Z/p$$ 
and
$$M^3 = H^{2p+5}(K(\Z,3)) \otimes H^{0}(BU_n)_{(p)} = \Zp \{x_{1}y_{p,0}\}/p\cong\Z/p.$$ 

\subsection*{The exactness of $\mathcal{M}$ at $M^2$}
\begin{lem}\label{lem:Td_2p-1}
In the spectral sequence $^TE$, we have
    \begin{equation}\label{eq:Td_3v_n}
        \begin{cases}
            v_{n}^{k}x_1 \in\opn{Im}{^Td}_3 ,\ 0 \le k \le p-2 \textrm{ or }k=p,\\
            ^{T}d^{3,*}_{2p-1}(v_{n}^{p-1}x_1) = y_{p,0}.
        \end{cases}
    \end{equation}
\end{lem}

\begin{proof}
    When $p \nmid k+1$,  the first formula in Proposition \ref{prop:diff0} together with  Proposition \ref{pro:diff1} imply that
    $$v_n^k x_1 = {\frac{1}{k+1}}{^Td}_3(v_n^{k+1})$$
    is in the image of ${^Td}_3$. This completes the proof for the case $0 \le k \le p-2 \textrm{ or }k=p$.
    
    The remaining case is proved by applying the second formula in Proposition \ref{prop:diff0}, taking 
    $e=l=1$, and then Proposition \ref{pro:diff1}.
\end{proof}

\begin{lem}\label{lem:image_of_delta1}
    The map $\delta^1:M^1\to M^2\cong\Z/p$ is surjective. 
\end{lem}
\begin{proof}
    Recall the morphism of fiber sequences $\Psi$ introduced in \eqref{eq:3_by_3_diag}, and the induced morphism $\Psi^*: {^U E}\to {^T E}$ of spectral sequences. 
    
    For $1\le i \le n$, let $v'_i =v_i-v_{n}$. It follows from (2) of Proposition \ref{prop:the_spectral_seq_TE} that the $v'_i$'s are permanent cycles.
    To determine the value of $\delta^1$ at $c_{p}x_1 \in M^1$, we have
    \begin{equation}\label{eq:Psi_delta1}
        \begin{split}
             &\Psi^*\delta^1(c_{p}x_1) \\
             =\ & \Psi^* ~ {^Ud}_{2p-1}^{3,2p}(c_{p}x_1) = {^Td}_{2p-1}^{3,2p}\Psi^{*}(c_{p}x_1)\\
             =\ & {^Td}_{2p-1}^{3,2p}(\sum_{n\geq i_1 > i_2 >...> i_p \geq 1.}v_{i_1}v_{i_2}...v_{i_p}x_1)\\
             =\ & {^Td}_{2p-1}^{3,2p}(\sum_{n\geq i_1 > i_2 >...> i_p\geq 1}(v'_{i_1}+v_n)(v'_{i_2}+v_n)...(v'_{i_p}+v_n)x_1)\\
             =\ & {^Td}_{2p-1}^{3,2p}(\sum_{n\geq i_1 > i_2 >...> i_p \geq 1}\sum_{j=0}^{p}\sigma_j(v'_{i_1},\cdots,v'_{i_p})v_n^{p-j}x_1)).
        \end{split}
    \end{equation}
    where $\Psi^*: {^UE}\to {^TE}$ is the morphism of spectral sequences induced by the inclusions of maximal tori $T^n\to U_n$ and $PT^n\to PU_n$, as introduced in (\ref{eq:3_by_3_diag}), and $\sigma_i$ the elementary symmetric polynomials in $p$ variables, as in \eqref{eq:sigma_def}.

    By Lemma \ref{lem:Td_2p-1}, we simplify \eqref{eq:Psi_delta1} and obtain 
    \begin{equation}\label{eq:Psi_delta1_second}
        \Psi^{*}\delta^1(c_{p}x_1) ={^Td}_{2p-1}(\sum_{n \geq i_1 > i_2 >...> i_p \geq 1}\sigma_1(v'_{i_1},\cdots,v'_{i_p})v_n^{p-1}x_1).
    \end{equation}
    To proceed, we evaluate the expression 
    $$\sum_{n\geq i_1 > i_2 >...> i_p \geq 1}\sigma_1(t_{i_1},\cdots,t_{i_p})$$
    for variables $t_i,\ 1\leq i\leq n$.  Since it is multi-linear and symmetric in the variables $t_1,\cdots,t_n$, we have
    \[\sum_{n\geq i_1 > i_2 >...> i_p \geq 1}\sigma_1(t_{i_1},\cdots,t_{i_p}) = \lambda\sum_{i=1}^nt_i\] 
    for some $\lambda\in\Z$. Taking the substitution $t_1=\cdots t_n=1$ and comparing both sides of the above, we obtain 
    \[\lambda=\frac{p}{n}\binom{n}{p}=\binom{n-1}{p-1}\not\equiv 0\pmod{p}\]
    and 
    \begin{equation}\label{eq:evaluate_sigma1}
        \sum_{n\geq i_1 > i_2 >...> i_p \geq 1}\sigma_1(t_{i_1},\cdots,t_{i_p}) = \binom{n-1}{p-1}\sum_{i=1}^nt_i.
    \end{equation}
   Consider the following commutative diagram:
    \begin{equation*}
        \begin{tikzcd}
            M^1 = {^UE}_2^{3,2p}\arrow[r,"\Psi^*"]\arrow[d] &{^TE}_{2}^{3,2p}\arrow[d] \\
             {^UE}_{2p-1}^{3,2p}\arrow[r,"\Psi^*"]\arrow[d,"^Ud_{2p-1}"] &{^TE}_{2p-1}^{3,2p}\arrow[d, "^Td_{2p-1}"] \\
           {^UE}_{2p-1}^{2p+2,2}\arrow[r,"\Psi^*"]\arrow[d, hook] &{^TE}_{2p-1}^{2p+2,2}\arrow[d, hook] \\
            M^2 = {^UE}_2^{2p+2,2}\arrow[r,"\Psi^*"] &{^TE}_{2}^{2p+2,2}
        \end{tikzcd}
    \end{equation*} 
    where the composition of the left vertical maps is $\de^1$ and we resume the computation of $\Psi^{*}\delta^1(c_{p}x_1)$ started in \eqref{eq:Psi_delta1_second}:
    \begin{equation}\label{eq:Psi_delta1_third}
        \begin{split}
             &\Psi^{*}\delta^1(c_{p}x_1)  \\
            =\ &{^Td}_{2p-1}(\binom{n-1}{p-1}\sum_{i=1}^n v_i'v_n^{p-1}x_1) \ \ \ (\textrm{by }\eqref{eq:evaluate_sigma1}) \\
            =\ &\binom{n-1}{p-1} \sum_{i=1}^n v'_i y_{p,0} \ \ (\textrm{since $v_i'$'s are permanent cocycles})\\
            =\ &\binom{n-1}{p-1}\sum_{i=1}^n v_i y_{p,0} \ \ (\text{since $y_{p,0}$ is $p$-torsion}) \\
            =\ &\Psi^{*}(\binom{n-1}{p-1}c_1y_{p,0}).
        \end{split}
    \end{equation}

    By the injectivity of  
    \[\Psi^{*}: M^2 = {^UE}_2^{2p+2,2}\to {^TE}_2^{2p+2,2}\]
    together with \eqref{eq:Psi_delta1_third}, we have 
    \[\delta^1(c_{p}x_1) = \binom{n-1}{p-1} c_{1}y_{p,0} \neq 0\]
    and we conclude.
\end{proof}

\begin{lem}\label{lem:exactness_at_M2}
    The chain complex $\mathcal{M}$ is exact at $M^2$.
\end{lem}
\begin{proof}
    By Lemma \ref{lem:image_of_delta1}, and the fact that $\mathcal{M}$ is a chain complex, we have $\delta^2 = 0$ and the lemma follows.
    
    Alternatively, one may compute $\delta^2 = d_3^{2p+2,2}$ directly with Corollary \ref{cor:d3} and obtain the same result.
\end{proof}

\subsection*{The exactness of $\mathcal{M}$ at $M^1$}
Recall that the $\Zp$-module $M^1$ is freely generated by elements of the form $cx_1$ for
\[c\in S' := \{c_1^{i_1}c_2^{i_2} \cdots c_p^{i_p} \mid i_k\geq 0 ,\ \sum_k ki_k = p\}.\]
Indeed, $S'$ is simply the set of monomials in $c_1,c_2\cdots,c_n$ in $H^{2p}(BU_n)$. We define a total ordering $\mathfrak{O}$ on monomials in $c_1,c_2\cdots,c_n$ as follows. We assert
\[c_1^{i_1}c_2^{i_2} \cdots c_p^{i_p} > c_1^{j_1}c_2^{j_2} \cdots c_p^{j_p}\]
if and only if 
\begin{enumerate}
    \item there is at least one $k$ such that $i_k\neq j_k$, and
    \item for the smallest such $k$, we have $i_k>j_k$.
\end{enumerate}
Let $S := S'- \{c_p\}$. Then $\mathfrak{O}$ defines total orderings on $S$, $S'$ and $S'x_1$ as well. To compare $cx_1,c'x_1\in S'x_1$, we assert $cx_1>c'x_1$ if and only if $c>c'$. 

Let $L$ be the $\Zp$-submodule of $H^{2p}(BU_n)\lp$ spanned by $S$. We define a $\Zp$-linear map 
\[\tau: L\to M^0 = H^{2p+2}(BU_n)\lp\]
as follows. Each element in $S$ is of the form $c_1^{i_1}c_2^{i_2}\cdots c_k^{i_k}$ such that $k<p$ and $i_k>0$, and we define
\[\tau(c_1^{i_1}c_2^{i_2}\cdots c_k^{i_k}) := (c_1^{i_1}c_2^{i_2}\cdots c_{k-1}^{i_{k-1}})(c_k^{i_k-1}c_{k+1}).\]

\begin{lem}\label{lem:delta0_tau}
    Let $\bar{\tau}:L/pL\to M^0/pM^0$ and $\bar{\delta}^0:M^0/pM^0\to M^1/pM^1$ denote the mod $p$ reductions of $\tau$ and $\delta^0$, respectively. Then the image of the composition 
    \[L/pL\xrightarrow{\bar{\tau}} M^0/pM^0 \xrightarrow{\bar{\delta}^0} M^1/pM^1\]
    is $Lx_1/pLx_1$. In particular, we have
    \begin{equation}\label{eq:delta_0_tau_image}
        \opn{Im}\delta^0\tau \subset W := Lx_1 + (pc_px_1) \subset M^1.
    \end{equation}
\end{lem}
\begin{proof}
    Consider the $\Zp$-basis $S$, $S'x_1$ for $L$ and $M^1$, respectively, both in the descending order with respect to the ordering $\mathfrak{O}$. Notice that $c_px_1$ is the smallest element in $S'$. We study the $(N+1)\times N$ matrix $A$ of the map 
    \[\delta^0\tau: L\to M^1\]
    with respect to these basis, where $N$ is the cardinality of $S$.
    
    Consider an arbitrary element 
    \[c:=c_1^{i_1}\cdots c_k^{i_k}\in S\]
    with $k<p$ and $i_k>0$. By Corollary \ref{cor:d3} and the Leibniz's formula, we have 
    \begin{equation*}
        \begin{split}
            & \delta^0\tau(c) = \delta^0(c_1^{i_1}\cdots c_{k-1}^{i_{k-1}}c_k^{i_k-1}c_{k+1})\\
            =&
            \begin{cases}
                (n-k)cx_1 + n i_1 c_1^{i_1-1}c_2^{i_2}\cdots c_k^{i_k-1}c_{k+1}x_1 + (\textrm{higher order terms}),\ i_1>0,\\
                (n-k)cx_1 + (\textrm{higher order terms}),\ i_1=0.
            \end{cases}
        \end{split}
    \end{equation*}
    In both cases, we have
    \begin{equation*}
        \delta^0\tau(c) \equiv (n-k)cx_1 + (\textrm{higher order terms})\pmod{p}.
    \end{equation*}
    Therefore, the matrix $A$ satisfies
    \begin{equation*}
        A \equiv
        \begin{pmatrix}
            \lambda_1 & *         & \cdots    & *       \\
                    0 & \lambda_2 &      *    & \vdots  \\
                    0 &         0 & \ddots    & *       \\
                    0 & \cdots    & 0         &\lambda_N\\
                    0 &         0 & \cdots    &0
        \end{pmatrix}\pmod{p},
    \end{equation*}
    where the $\lambda_i$'s are of the form $n-k$ for $k<p$, which are invertible in $\Zp$, and we have verified that the image of the composition 
    \[L/pL\xrightarrow{\bar{\tau}} M^0/pM^0 \xrightarrow{\bar{\delta^0}} M^1/pM^1\]
    is $Lx_1/pLx_1$. The equation \eqref{eq:delta_0_tau_image} follows from the above and the fact 
    \[M^1 = Lx_1 + (c_px_1).\]
\end{proof}

\begin{lem}\label{lem:V_to-W}
    Consider the $\Zp$-submodule $V=\tau(L)+(c_1c_p-c_{p+1})$ of $M^0$. We have $\delta^0(V)\subset W$ where 
    \[W := Lx_1 + (pc_px_1)\subset M^1\]
    is the $\Zp$-submodule of $M^1$ defined in Lemma \ref{lem:delta0_tau}.
\end{lem}
\begin{proof}
    By Lemma \ref{lem:delta0_tau} we have $\delta^0(\tau(L)) \subset W$. On the other hand, we have
    \begin{equation}\label{eq:c_1c_p}
        \delta^0(c_1c_p-c_{p+1}) = (n-p+1)c_1c_{p-1}x_1 + pc_px_1\in W,
    \end{equation}
    and we conclude.
\end{proof}

\begin{lem}\label{lem:exact_at_M1}
    The chain complex $\mathcal{M}$ is exact at $M^1$.
\end{lem}
\begin{proof}
    By Lemma \ref{lem:V_to-W}, the restriction of $\delta^0$ to $V$ has image in $W$. Therefore, we write $\delta^0_V := \delta^0|_V: V\to W$ and consider its mod $p$ reduction \[\bar{\delta}^0_V: V/pV\to W/pW = Lx_1/pLx_1 + (pc_px_1)/(p^2c_px_1).\]
    By Lemma \ref{lem:delta0_tau}, we have $Lx_1/pLx_1\subset \opn{Im}{\bar{\delta}^0|_V}$. 
    
    By $Lx_1/pLx_1\subset \opn{Im}{\bar{\delta}^0_V}$ and \eqref{eq:c_1c_p}, we have $[pc_px_1]\in \opn{Im}{\bar{\delta}^0_V}$, where $[pc_px_1]$ is the class in $W/pW$ represented by $pc_px_1$. Therefore, $\bar{\delta}^0_V: V/pV\to W/pW$ is surjective. By Nakayama's lemma in commutative algebra (Thoerem 2.2, Chapter 1, \cite{matsumura1989commutative}), $\delta^0_V: V\to W$ is surjective.
    
    Therefore, we have
    \begin{equation}\label{eq:exact_at_M1_1}
        \opn{Im}\delta^0 \supset\opn{Im}\delta^0_V = W = Lx_1 + (pc_px_1).
    \end{equation}
    On the other hand, we have $\opn{Ker}\delta^1 \supset \opn{Im}\delta^0$, and therefore $\opn{Ker}\delta^1 \supset W$. Now, by Lemma \ref{lem:image_of_delta1}, we have 
    \[\Z/p\cong M^1/(L + (pc_px_1)) = M^1/W \to M^1/\opn{Ker}\delta^1\cong\Z/p,\]
    where the arrow is the tautological quotient map, which is surjective. Therefore, the above composition is a bijection. It follows that we have
    \begin{equation}\label{eq:exact_at_M1_2}
        W = \opn{Ker}\delta^1 \supset \opn{Im}\delta^0,
    \end{equation}
    and the lemma follows from \eqref{eq:exact_at_M1_1} and \eqref{eq:exact_at_M1_2}.
\end{proof}

Lemma \ref{lem:exact_at_M1} and Lemma \ref{lem:exactness_at_M2} complete the proof of Proposition \ref{prop: four term exact sequence}.

\bibliographystyle{plain}
\bibliography{RefpPrimaryBPU}

\end{document}